\documentclass[journal]{IEEEtran}

\usepackage{graphicx}      
\usepackage{amsmath} 
\usepackage{amssymb}  
\usepackage{amsthm}
\usepackage{color,soul}
\usepackage{colortbl}
\usepackage{subfigure}
\usepackage{array}
\usepackage{booktabs}
\usepackage{url}
\usepackage{multicol}
\usepackage{multirow}
\usepackage{epstopdf}
\usepackage{makecell}
\usepackage{enumerate}

\newtheorem{lemma}{Lemma}
\newtheorem{clr}{Corollary}
\newtheorem{thm}{Theorem}
\newtheorem{defi}{Definition}
\newtheorem{ass}{Assumption}

\theoremstyle{remark}
\newtheorem{rem}{Remark}

\begin{document}
%
\title{A Posteriori Probabilistic Bounds of Convex Scenario Programs with Validation Tests}
%
%
%

\author{Chao~Shang,~\IEEEmembership{Member,~IEEE}
        and~Fengqi~You,~\IEEEmembership{Senior Member,~IEEE}
\thanks{

C. Shang is with Department of Automation, Beijing National Research Center for Information Science and Technology, Tsinghua University, Beijing 100084, China (e-mail: {\tt c-shang@tsinghua.edu.cn}).

F. You is with College of Engineering, Cornell University, Ithaca, New York 14853, USA (e-mail: {\tt fengqi.you@cornell.edu}).}
}

\maketitle

\begin{abstract}
Scenario programs have established themselves as efficient tools towards decision-making under uncertainty. To assess the quality of scenario-based solutions a posteriori, validation tests based on Bernoulli trials have been widely adopted in practice. However, to reach a theoretically reliable judgement of risk, one typically needs to collect massive validation samples. In this work, we propose new a posteriori bounds for convex scenario programs with validation tests, which are dependent on both realizations of support constraints and performance on out-of-sample validation data. The proposed bounds enjoy wide generality in that many existing theoretical results can be incorporated as particular cases. To facilitate practical use, a systematic approach for parameterizing a posteriori probability bounds is also developed, which is shown to possess a variety of desirable properties allowing for easy implementations and clear interpretations. By synthesizing comprehensive information about support constraints and validation tests, improved risk evaluation can be achieved for randomized solutions in comparison with existing a posteriori bounds. Case studies on controller design of aircraft lateral motion are presented to validate the effectiveness of the proposed a posteriori bounds.
\end{abstract}


\IEEEpeerreviewmaketitle

\section{Introduction}
\label{sec:introduction}
\IEEEPARstart{T}HE widespread presence of uncertainty has been invariably a crucial issue in design, analysis and optimization of complex systems, and the ignorance of uncertainty can lead to decisions that are fragile in the real-world uncertain environment. In the past decades, decision-making under uncertainty has raised immense research efforts across various communities. Typical examples include robust optimal control, where control inputs are designed to yield invariably satisfactory performance under model mismatch, unmeasured disturbance, and measurement noise \cite{zhou1996robust,tempo2012randomized}. Likewise, in filter design for fault detection, such effects have been taken into account in order to avoid high false alarm rates \cite{ding2008model}.


The increasing availability of data has motivated rapid development of data-driven methods in systems and control \cite{van2020data}. As a representative, the scenario approach has been widely adopted as a data-driven approach for reliable decision-making under uncertainty \cite{calafiore2005uncertain,campi2008exact}, where a finite number of constraints defined on past samples of uncertainty are enforced to seek robustness. It has found wide applications in operations of smart grid \cite{wang2012chance}, service systems \cite{deng2016decomposition}, supply chain management \cite{santoso2005stochastic,you2013multicut}, and control design \cite{calafiore2006scenario,calafiore2011research}. The scenario approach is known to generalize well, which has been revealed theoretically that the probability of constraint violation of the solution can be safely controlled with high confidence by choosing an adequately large sample size \cite{calafiore2005uncertain,campi2008exact}. To trade robustness for performance, extensions based on constraint discarding have been also developed \cite{luedtke2008sample,campi2011sampling,care2015scenario}.

Due to the inherent randomness of scenario sampling, the optimal solution of a scenario program is itself a random variable; hence, the risk of the randomized solution, especially the violation probability, is a fortiori uncertain. Therefore, for robust and secure decision-making, a fair evaluation of the risk of the randomized solution must be performed before its implementation in real-world situations. If the risk is considered to be too high, the solution will not be accepted, and further refinement is necessitated. Towards this goal, the simplest approach is to carry out validation test on a collection of new instances of the uncertainty. If few violations are seen on the validation set, the solution is believed to offer a high protection level, and vice versa. Technically, the probabilistic guarantee is established based on finite-sample bounds for Bernoulli trials. The validation strategy is also adopted in the so-called \textit{sequential scenario approach} originally developed by \cite{oishi2007polynomial,dabbene2010randomized}, which solves a sequence of reduced-size scenario programs that are more computationally affordable \cite{care2014fast,alamo2015randomized,chamanbaz2016sequential,calafiore2017repetitive}. After obtaining the candidate solution in each iteration, validation test is performed to verify whether its violation probability is smaller than a pre-specified threshold. If so, the sequential algorithm terminates and a final solution is returned.


Recently, an emerging line of research has concentrated on the so-called ``wait-and-judge" scenario approach \cite{campi2018wait,care2019wait,campi2018general,garatti2019risk}. To attain a robust design, sufficient data samples are first incorporated for optimization, and \textit{a posteriori} assessment of the reliability of scenario-based solutions can be made by counting observed decisive constraints, which can be interpreted as the \textit{complexity} of the scenario program \cite{garatti2019risk}. In comparison with a priori violation probability bounds, significant improvement can be achieved by the wait-and-judge approach, and it is even unnecessary to upper-bound the number of support constraints in advance \cite{campi2018wait,care2019wait}. However, this scheme still has some limitations in practical use. First, there may be computational limits that prohibit the utilization of all available data in the scenario program to seek robustness \cite{you2018distributed}. To alleviate the computational burden, one has to use only a fraction of scenarios for optimization and give up valuable information within unused scenarios. Second, after implementing a scenario-based decision new instances of uncertainty could keep arriving, and one may want to use these out-of-sample data to refine the risk judgement of the decision in hand. In both cases, additional information in scenarios that are not involved in the design problem cannot be effectively utilized for risk evaluation by the present wait-and-judge theory.

In this article, we seek to fill the knowledge gap by establishing a general class of \textit{a posteriori probabilistic bounds} for scenario-based solutions, where a decision-maker has to solve a scenario program with a fixed sample size, and then further evaluates the solution's risk on out-of-sample realizations of uncertainty. The proposed probabilistic bounds depend simultaneously on the number of decisive support constraints and the outcome of Bernoulli trials. In this way, refined evaluation of solution's risk can be attained by leveraging an increased amount of information from both decisive support constraints and out-of-sample performance on validation data. The established probabilistic bounds turn out to subsume many existing a priori and a posteriori probabilistic bounds as special cases, thereby enjoying widespread generality. To make practical use of the proposed bounds, an efficient parameterization approach is developed. By incorporating information from validation tests, the proposed bounds turn out to be no higher than existing wait-and-judge bounds that only use partial information. We also show that they bear clear interpretations because of their consistency with the validation performance, thereby greatly enhancing the practicability of wait-and-judge scenario optimization. In addition, we establish fundamental lower limits for the proposed bounds, and develop a refinement procedure to reduce the conservatism. Finally, we illustrate our main results with case studies on linear quadratic regulator (LQR) design of aircraft lateral motion, where the proposed bounds provide empirically better judgements than both the wait-and-judge approach and the outcome of Bernoulli trials. Comprehensive comparisons indicate that the efficacy of proposed bounds owes to a sophisticated synthesis of support constraint information and validation test information. For small validation sample size, it mainly relies on support constraint information and largely improves upon bounds based on Bernoulli trials, while under large validation sample size, validation test information tends to have a dominating effect and contribute to reliable judgements.

The layout of this paper is organized as follows. In Section \ref{sec:2}, the technical background of scenario programs is introduced. Section \ref{sec:3} formally states the main results of this work, and discusses some practical implementation issues. An illustrative example is provided in Section \ref{sec:5}, followed by concluding remarks in the last section.

\textit{Notations and Definitions}. $\mathbb{N}_0$ is the set of non-negative integer numbers, and the set of consecutive non-negative integers $\{j,\cdots,k\}$ is denoted by $\mathbb{N}_{j:k}$. The $p$-norm of a matrix is denoted by $\|\cdot\|_p$, while $\|\cdot\|$ represents the Euclidean norm of a vector by convention. The trace operator on a square matrix is denoted by ${\rm Tr}\{\cdot\}$. The indicator function of a subset $\mathcal{S}$ is defined as $\bold{1}_{\mathcal{S}}(s)$. Given a positive integer $N$, a non-negative integer $m \le N$, and $t \in (0,1)$, the Beta cumulative probability function is defined as $B_N(t;m) := \sum_{i=0}^m \binom{N}{i}t^i(1-t)^{N-i}$. $\bold{C}^q[0,1]$ denotes the class of $q$ times differentiable functions with continuous $q$-th derivative over the interval $[0,1]$, and $\bold{P}^K$ denotes the set of polynomials of degree $K$.

\section{Preliminaries}
\label{sec:2}
\subsection{Scenario Programs}
Assume that $\Delta$ is a probability space, which is associated with a $\sigma$-algebra $\mathcal{F}$ and a probability measure $\mathbb{P}$. $\delta$ is a random outcome from the triplet $(\Delta,\mathcal{F},\mathbb{P})$. Moreover, denote by $(\Delta^m,\mathcal{F}^m,\mathbb{P}^m)$ the Cartesian product of $\Delta$ equipped with the $\sigma$-algebra $\mathcal{F}^m$ and the $m$-fold probability measure $\mathbb{P}^m \doteq \mathbb{P} \times \cdots \times \mathbb{P}$ ($m$ times). In general, the analytic form of $\mathbb{P}$ is unknown in practice, but we are allowed to sample a set of independent scenarios $\{ \delta^{(1)},\cdots,\delta^{(m)} \}$ from $\mathbb{P}^m$. Here, we refer to $\omega_m := \{ \delta^{(1)},\cdots,\delta^{(m)} \} \in \Delta^m$ as a \textit{multi-sample}, based on which the following scenario programs can be defined:
\begin{equation*}
{\rm SP}_m[\omega_m]: \left \{ \begin{split}
\min_{\bold{x} \in \mathbb{R}^d} &\ J(\bold{x}) \\
{\rm s.t.} &\ \bold{x} \in \mathcal{X} \subseteq \mathbb{R}^d \\
           &\ f(\bold{x}, \delta^{(i)}) \le 0,\ \forall i \in \mathbb{N}_{1:m}
\end{split} \right .
\label{eq:sp}
\end{equation*}
when $m > 0$, and
\begin{equation*}
{\rm SP}_0[\omega_0]: \left \{ \begin{split}
\min_{\bold{x} \in \mathbb{R}^d} &\ J(\bold{x}) \\
{\rm s.t.} &\ \bold{x} \in \mathcal{X} \subseteq \mathbb{R}^d
\end{split} \right .
\label{eq:sp0}
\end{equation*}
with $m = 0$ as a degenerate case. In this work, we focus on convex scenario programs, where as regularity conditions, the set $\mathcal{X}$ is assumed to be convex and closed, the objective function $J(\bold{x})$ is convex, and $f(\bold{x}, \delta)$ is lower semi-continuous and convex in $\bold{x} \in \mathbb{R}^d$ with the value of $\delta$ fixed. Given a candidate decision $\bold{x} \in \mathcal{X}$, its \textit{violation probability} is defined as:
\begin{equation}
V(\bold{x}) = \mathbb{P} \left \{ \delta \in \Delta: f(\bold{x}, \delta) > 0
\right \},
\end{equation}
which can also be interpreted as the \textit{risk} of $\bold{x}$, and the corresponding \textit{reliability level} is $1 - V(\bold{x})$ \cite{tempo2012randomized}.

Meanwhile, the following assumption is made throughout the paper, which is fairly standard in generic convex scenario programs \cite{campi2008exact,calafiore2006scenario}.
\begin{ass}[Feasibility and Uniqueness]
\label{ass:feas}
For every $m \in \mathbb{N}_0$ and every multi-sample $\omega_m \in \Delta^m$, the optimal solution $\bold{x}^*_m$ to the scenario program ${\rm SP}_m[\omega_m]$ exists and is unique.
\end{ass}

Note that in order to secure the uniqueness of $\bold{x}_m^*$, a suitable tie-break rule can be employed \cite{calafiore2005uncertain}. In this way, the optimal solution $\bold{x}_m^*$ to ${\rm SP}_m[\omega_m]$ is essentially a random variable defined over $\Delta^m$. Under Assumption \ref{ass:feas} and the conventional assumption on the measurability of $V(\bold{x}_m^*)$, a celebrated probabilistic guarantee in literature is that, the tail probability of $V(\bold{x}_m^*)$ is dominated by the Beta cumulative probability function \cite{campi2008exact}:
\begin{equation}
\mathbb{P}^m \left \{ V(\bold{x}_m^*) > \epsilon \right \} \le B_m(\epsilon;\zeta-1),\ m > \zeta,
\label{eq:campi}
\end{equation}
where $\zeta (\le d)$ can be the Helly's dimension \cite{calafiore2010random} or its tighter bound built with structural information \cite{schildbach2013randomized,zhang2015sample}. The probabilistic guarantee (\ref{eq:campi}) is closely related to the concept of \textit{support constraints}, which will be made heavy use of in this paper. \par

\begin{defi}[Support Constraint, \cite{campi2008exact}]
A constraint of the scenario program ${\rm SP}_m[\omega_m]$ is a support constraint if its removal yields a different optimal solution of the initial problem.
\end{defi}

Upon solving ${\rm SP}_m[\omega_m]$, the indices of induced support constraints can be described by a set $I_m^* \subseteq \mathbb{N}_0$, which is a set-valued mapping from the multi-sample $\omega_m$. The cardinality of $I_m^*$, i.e. the associated number of support constraints, is given by $s_m^* = |I_m^*|$. Likewise, $s_m^*$ is a function of $\omega_m$; hence, both $I_m^*$ and $s_m^*$ are essentially random variables on the triplet $(\Delta^m,\mathcal{F}^m,\mathbb{P}^m)$. The definition of support constraints differs from that of generic \textit{active constraints}, which are characterized by $\{i:\ f(\bold{x}_m^*, \delta^{(i)}) = 0 \}$. For convex scenario programs, support constraints are always active constraints, but the converse no longer remains true \cite{campi2018wait}. Meanwhile, under the availability of sufficient scenarios, the number of support constraints is always upper bounded by $\zeta$ \cite{calafiore2005uncertain}. For a fully-supported problem where the number of support constraints is exactly $\zeta$ with probability one (w.p.1), the inequality (\ref{eq:campi}) then becomes an equality \cite{campi2008exact}; in other cases, it inevitably induces conservatism.

Another useful assumption is made as follows.

\begin{ass}[Non-Degeneracy, \cite{campi2018wait}]
\label{ass:nondeg}
The solution to the scenario program ${\rm SP}_m[\omega_m]$ coincides w.p.1 with the solution to the program defined by support constraints only.
\end{ass}

It is noteworthy that Assumption \ref{ass:nondeg} is quite mild for general convex scenario programs, because it excludes anomalous situations where the solution defined by support constraints lies exactly on boundaries of other constraints with a nonzero probability.

\subsection{Confidence Intervals for Violation Probability with Validation Tests}

Given a candidate solution $\bold{x}$, a usual approach to evaluate its violation probability $V(\bold{x})$ is to perform Bernoulli trials with some independent validation scenarios $\{ \tilde{\delta}^{(1)},\cdots,\tilde{\delta}^{(M)}\}$. There are two possible outcomes in each trial, i.e. ``success" and ``failure", which are characterized by $f(\bold{x}, \tilde{\delta}^{(l)}) \le 0$ and $f(\bold{x}, \tilde{\delta}^{(l)}) > 0$, respectively. By summarizing results of Bernoulli trials, the empirical violation frequency on the validation dataset can be obtained.

\begin{defi}[Empirical Violation Frequency]
Given a set of $M$ independent scenarios $\{ \tilde{\delta}^{(1)},\cdots,\tilde{\delta}^{(M)}\}$, the empirical violation frequency of a solution $\bold{x}$ is calculated as:
\begin{equation}
r_M = \sum_{l=1}^M \bold{1}_{ \{\delta:f(\bold{x}, \delta) > 0 \} } (\tilde{\delta}^{(l)} ).
\label{eq:freq}
\end{equation}
That is, the number of scenario-based validation constraints that are violated by $\bold{x}$.
\end{defi}

With $\bold{x}$ fixed, $r_M$ is a random outcome on $\mathbb{P}^M$, which embodies useful information for estimating the binomial proportion, i.e., the true violation probability $V(\bold{x})$. In the context of robust design, a small violation probability is always desired, and hence an upper-bound for $V(\bold{x})$ based on $r_M$ shall be constructed to quantify the risk of an unreliable solution. This can be achieved by adopting the following one-sided bounds for evaluating the binomial proportion from a series of Bernoulli trials \cite{tempo2012randomized}.

\begin{thm}[One-Sided Chernoff Bound, \cite{chernoff1952measure}]
\label{thm:00}
Given a solution $\bold{x}$ and a pre-defined confidence level $\beta^* \in (0,1)$, it then holds that:
\begin{equation}
\mathbb{P}^{M} \left \{ V(\bold{x}) > \rho_M(r_M)  \right \} \le \beta^*,
\label{eq:onesided_Chernoff}
\end{equation}
where the a posteriori violation probability $\rho_M(r_M) = r_M/M + \sqrt{\log \beta^*/(-2M)}$ is a random variable depending on the realization of $r_M$.
\end{thm}

\begin{thm} [One-Sided Clopper-Pearson (C-P) Bound, \cite{clopper1934use,conover1980practical,thulin2014cost}]
\label{thm:0}
Given a solution $\bold{x}$ and a pre-defined confidence level $\beta^* \in (0,1)$, it then holds that:
\begin{equation}
\mathbb{P}^{M} \left \{ V(\bold{x}) > \eta_{M}(r_M) \right \} \le \beta^*,
\label{eq:onesided}
\end{equation}
where the analytic form of $\eta_M(\cdot)$ is expressed as:
\begin{equation}
\eta_{M}(l) = \left \{
\begin{array}{l l}
\min_{\eta} \{\eta:\ B_M(\eta;l) \le \beta^* \}, & {\rm if}\ l \ne M \\
1, & {\rm if}\ l=M
\end{array}
\right .
\label{eq:pmk}
\end{equation}
\end{thm}

The above probabilistic bounds are referred to as \textit{a posteriori bounds}, in the sense that they provide certificates on $V(\bold{x})$ after the value of $r_M$ is revealed \cite{tempo2012randomized}. Their use in validating scenario-based solutions can be illustrated as follows. Considering the case where there are $M=100$ samples for validating $\bold{x}$, and $r_M$ turns out to be $10$, setting $\beta^* = 10^{-6}$, one then obtains the one-sided Chernoff bound $\rho_{100}(10) = 0.3628$ and the one-sided C-P bound $\eta_{100}(10) = 0.3045$, respectively, indicating that the violation probability $V(\bold{x}_N^*)$ does not exceed $36.28\%\ ({\rm or}\ 30.45\%)$ with a very high confidence of $1 - 10^{-6}$ (i.e. $99.9999\%$). Note that the calculation of $\rho_M(r_M)$ is easier than that of $\eta_M(r_M)$. However, it is known that $\rho_M(r_M)$ is generally more conservative than $\eta_M(r_M)$, and it is possible that $\rho_M(r_M)$ have values greater than one, especially when the value of $r_M$ is large; by contrast, $\eta_M(r_M)$ is always between zero and one, thereby enjoying better practicability.

\section{Main Results}
\label{sec:3}
\subsection{A Posteriori Confidence Level Based on Violation Probability Function}

Assume that we are seated in a situation where a scenario program ${\rm SP}_N[\omega_N]$ with fixed $N>\zeta$ ($N$ is the initial number of available scenarios for optimization), is first solved to pursue robustness, and then the risk of $\bold{x}^*_N$ is further evaluated by computing the empirical violation frequency $r_M^*$ on $M$ extra validation samples. Since both the one-sided Chernoff bound and the one-sided C-P bound are applicable to any candidate solution $\bold{x}$, we can certainly use them to evaluate $\bold{x}^*_N$'s risk as a function of $r_M^*$. However, the optimal solution $\bold{x}_N^*$ as well as $s_N^*$ have some inherent randomness, whose realizations embody useful information about reliability. In this section, we present a more general theory to concurrently tackle two classes of randomness. To be more specific, the a posteriori bound $\epsilon(s_N^*,r_M^*)$ on the violation probability $V(\bold{x}^*_N)$ is considered as a function of both $s_N^*$ and $r_M^*$. With a group of functions $\{ \epsilon(k,l) \}$ pre-specified for $k \in \mathbb{N}_{0:\zeta},\ l \in \mathbb{N}_{0:M}$, the following finite-sample probabilistic guarantee can be established as one of our main results.

\begin{thm}
\label{thm:1}
Assume that $\{ \delta^{(1)}, \cdots, \delta^{(N)}, \tilde{\delta}^{(1)}, \cdots, \tilde{\delta}^{(M)} \}$ are sampled from the probability space $(\Delta^{N+M},\mathcal{F}^{N+M},\mathbb{P}^{N+M})$, and $\epsilon(k,l)$ is an arbitrary $[0,1]$-valued function where $k \in \mathbb{N}_{0:\zeta},\ l \in \mathbb{N}_{0:M}$. Then it holds that
\begin{equation}
\mathbb{P}^{N+M} \left \{ V(\bold{x}_N^*) > \epsilon(s_N^*,r_M^*) \right \} \le \beta^*,
\label{eq:guarantee}
\end{equation}
where $\beta^*$ is the optimal value of the following variational problem, defined based on functions $\{ \epsilon(k,l) \}$:
\begin{equation}
\begin{split}
&\ \inf_{g \in \bold{C}^{\zeta}[0,1]} g(1) \\
&\ \ \ \ \ {\rm s.t.} \ \ \ \binom{N}{k} t^{N-k} \cdot \sum_{l=0}^M  \binom{M}{l} t^{M-l} (1-t)^{l} \cdot \bold{1}_{[0,1-\epsilon(k,l))} (t) \\
&\ \ \ \qquad \qquad \quad \quad \quad \quad \le \frac{1}{k!}\frac{{\rm d}^k}{{\rm d}t^k}g(t),\ \forall t \in [0,1],\ k \in \mathbb{N}_{0:\zeta}
\end{split}
\label{eq:beta}
\end{equation}
\end{thm}

\begin{proof} The proof bears resemblance to Theorem 1 in \cite{campi2018wait} so we only provide details for key steps herein. First, the following group of \textit{generalized distribution functions} is defined:
\begin{equation}
F_k(v) = \mathbb{P}^k \left \{ V(\bold{x}_k^*) \le v \land s_k^* = k \right \},\ k\in \mathbb{N}_{0:\zeta}.
\label{eq:generalizedDF}
\end{equation}
Next, we show that the a posteriori violation probability $\mathbb{P}^{N+M} \left \{ V(\bold{x}_N^*) > \epsilon(s_N^*,r_M^*) \right \}$ can be calculated based on $\{ F_k(v) \}_{k=0}^{\zeta}$, which act as the backbone of the whole machinery. Notice that the following events do not overlap with each other:
\begin{equation*}
\left \{ s_N^* = k \land r_M^* = l \right \} \bigcap \left \{ s_N^* = k' \land r_M^* = l' \right \} = \emptyset,
\end{equation*}
if $k \ne k'$ or $l \ne l'$. Then the following decomposition can be made:
\begin{equation*}
\begin{split}
&\ \mathbb{P}^{N+M} \left \{ V(\bold{x}_N^*) > \epsilon(s_N^*,r_M^*) \right \} \\
= &\ \sum_{k=0}^{\zeta} \sum_{l=0}^M \mathbb{P}^{N+M} \left \{ V(\bold{x}_N^*) > \epsilon(k,l) \land s_N^* = k \land r_M^* = l \right \}.
\end{split}
\end{equation*}
We define the events on the probability space $\Delta^{N+M}$:
\begin{equation*}
\mathcal{V}_{k,l} = \left \{ V(\bold{x}_N^*) > \epsilon(k,l) \land s_N^* = k \land r_M^* = l \right \} \subseteq \Delta^{N+M}.
\end{equation*}
which indicates that the violation probability of $\bold{x}_N^*$ is larger than $\epsilon(k,l)$, and there are $k$ support constraints and $l$ violated validation constraints, respectively. In this way,
\begin{equation}
\mathbb{P}^{N+M} \left \{ V(\bold{x}_N^*) > \epsilon(s_N^*,r_M^*) \right \} = \sum_{k=0}^{\zeta}\sum_{l=0}^M \mathbb{P}^{N+M} \left \{ \mathcal{V}_{k,l} \right \}
\label{eq:9}
\end{equation}
holds. Next we derive analytic expressions for $\mathbb{P}^{N+M} \left \{ \mathcal{V}_{k,l} \right \}$. By fixing the indices of support constraints, we have:
\begin{equation}
\mathbb{P}^{N+M}\{\mathcal{V}_{k,l}\} = \binom{N}{k} \mathbb{P}^{N+M}\{\mathcal{A}_1 \cap \mathcal{A}_2\}.
\label{eq:10}
\end{equation}
where
\begin{gather*}
\begin{split}
\mathcal{A}_1 = & \left \{ V(\bold{x}_N^*) > \epsilon(k,l)\ \land\ s_N^* = k\ \land\ I_N^* = \mathbb{N}_{1:k} \right \},
\end{split} \\
\mathcal{A}_2 = \left \{ r_M^* = l \right \}.
\end{gather*}
It has been proved by \cite{campi2018wait} that, in the probability space $(\Delta^N,\mathcal{F}^N,\mathbb{P}^N)$, the following events are equal up to a zero probability set:
\begin{equation*}
\begin{split}
&\ \mathcal{A}_1 = \mathcal{B}_1 \triangleq \{ V(\bold{x}_k^*) > \epsilon(k,l) \land s_k^* = k \\
&\ \quad \quad \quad \quad \quad \quad \quad \quad \land f(\bold{x}_k^*,\delta^{(j)}) \le 0,\ j \in \mathbb{N}_{k+1:N} \},
\end{split}
\end{equation*}
which implies that:
\begin{equation}
\begin{split}
&\ \mathbb{P}^{N+M}\{\mathcal{A}_1 \cap \mathcal{A}_2\} \\
= &\ \mathbb{P}^{N+M}\{\mathcal{B}_1 \cap \mathcal{A}_2\} \\
= &\ \mathbb{P}^{N+M} \left \{ \begin{split}
&\ V(\bold{x}_k^*) > \epsilon(k,l) \land s_k^* = k  \land r_M^* = l \\
&\ \land f(\bold{x}_k^*,\delta^{(j)}) \le 0,\ j \in \mathbb{N}_{k+1:N} \end{split} \right \} \\
= &\ \int_{(\epsilon(k,l),1]} \binom{M}{l} (1-v)^{N-k+M-l} v^{l} {\rm d}F_k(v)
\end{split}
\label{eq:14}
\end{equation}
where the last equality is due to the fact that by fixing violation probability $v$, $(1-v)^{N-k}$ is the probability that $\mathcal{X}_{\delta^{(k+1)}}, \cdots, \mathcal{X}_{\delta^{(N)}}$ are satisfied by $\bold{x}_k^*$, and $\binom{M}{l}v^l(1-v)^{M-l}$ is the probability of observing exactly $l$ violated validation constraints. By using the definition of $F_k(v)$ in (\ref{eq:generalizedDF}) and taking integral over the interval $(\epsilon(k,l),1]$, (\ref{eq:14}) can be attained. Substituting (\ref{eq:10}) and (\ref{eq:14}) into (\ref{eq:9}) yields:
\begin{equation}
\begin{split}
&\ \mathbb{P}^{N+M} \left \{ V(\bold{x}_N^*) > \epsilon(s_N^*,r_M^*) \right \} \\
= &\ \sum_{k=0}^{\zeta} \sum_{l=0}^M \binom{N}{k}\binom{M}{l} \int_{(\epsilon(k,l),1]} (1-v)^{N-k+M-l} v^{l} {\rm d}F_k(v).
\label{eq:16}
\end{split}
\end{equation}

Next we show that the RHS of (\ref{eq:16}) is actually upper-bounded by $\beta^*$ given in (\ref{eq:beta}). The idea is to take generalized distribution functions $\{ F_k(v) \}_{k=0}^{\zeta}$ as decision variables subject to the following joint moment conditions \cite{campi2018wait}:
\begin{equation*}
\sum_{k=0}^{\min\{m,\zeta\}} \binom{m}{k} \int_{[0,1]} (1-v)^{m-k} {\rm d}F_k(v) = 1,\ m \in \mathbb{N}_0.
\end{equation*}
Therefore, we have:
\begin{equation*}
\mathbb{P}^{N+M} \left \{ V(\bold{x}_N^*) > \epsilon(s_N^*,r_M^*) \right \} \le \beta',
\end{equation*}
where $\beta'$ is the optimal value of the following variational problem:
\begin{equation}
\begin{split}
\sup_{\{ F_k(\cdot) \}} & \sum_{k=0}^{\zeta} \sum_{l=0}^M \binom{N}{k}\binom{M}{l} \int_{(\epsilon(k,l),1]} (1-v)^{N-k+M-l} v^{l} {\rm d}F_k(v) \\
{\rm s.t.}\ & \sum_{k=0}^{\min\{m,\zeta\}} \binom{m}{k} \int_{[0,1]} (1-v)^{m-k} {\rm d}F_k(v) = 1,\ m \in \mathbb{N}_{0} \\
& \ \ F_0, F_1, \cdots, F_{\zeta} \in \mathcal{C}
\end{split}
\label{eq:19}
\end{equation}
Here, $\mathcal{C}$ stands for the set of all positive generalized distribution functions. However, (\ref{eq:19}) is a generalized moment problem with infinite moment constraints, making the optimization problem difficult to solve analytically. Hence we turn to the following truncated problem:
\begin{equation}
\begin{split}
\sup_{\{ F_k(\cdot) \}} & \sum_{k=0}^{\zeta} \sum_{l=0}^M \binom{N}{k}\binom{M}{l} \int_{(\epsilon(k,l),1]} (1-v)^{N-k+M-l} v^{l} {\rm d}F_k(v) \\
{\rm s.t.}\ &\sum_{k=0}^{\min\{m,\zeta\}} \binom{m}{k} \int_{[0,1]} (1-v)^{m-k} {\rm d}F_k(v) = 1,\ m \in \mathbb{N}_{0:K} \\
& \ F_0, F_1, \cdots, F_{\zeta} \in \mathcal{C}
\end{split}
\label{eq:20}
\end{equation}
whose optimal value is $\beta'_K$. Because only $K$ constraints are involved in (\ref{eq:20}), (\ref{eq:20}) is a relaxation of (\ref{eq:19}), so we have $\beta'_K \ge \beta',\ \forall K$. Meanwhile, the value of $\beta'_K$ is non-increasing with $K$. By deriving the Lagrangian function and optimizing over $\{ F_k(\cdot) \}_{k=0}^{\zeta}$, we obtain the following dual problem \cite{anderson1987linear}:
\begin{equation}
\begin{split}
\inf_{\lambda_0,\cdots,\lambda_K} &\ \sum_{m=0}^K \lambda_m \\
{\rm s.t.} \ \  & \binom{N}{k} \sum_{l=0}^M \binom{M}{l} (1-v)^{N-k+M-l} v^{l} \bold{1}_{(\epsilon(k,l),1]} (v) \\
&\  \le \sum_{m=k}^K \lambda_m \binom{m}{k} (1-v)^{m-k},\ \forall t \in [0,1],\ k \in \mathbb{N}_{0:\zeta}
\end{split}
\label{eq:21}
\end{equation}
where $\lambda_0,\cdots,\lambda_K$ are dual variables and $\beta^*_K$ is the optimal value. Therefore, it holds that $\beta' \le \beta'_K \le \beta_K^*,\ \forall K \in \mathbb{N}_0$ due to weak duality. We define the polynomial of degree $K$ as $g(t) = \sum_{m=0}^K \lambda_m t^m$, which admits the following properties:
\begin{equation}
g(1) = \sum_{m=0}^K \lambda_m,\ \frac{1}{k!} \frac{{\rm d}^k}{{\rm d}t^k} t^m = \left \{ \begin{split}
& 0, \ \ \ \quad \quad \ \ \ \ m < k \\
& \binom{m}{k}t^{m-k}, \ m \ge k
\end{split}
\right .
\label{eq:22}
\end{equation}
By defining $t := 1 - v$ and plugging (\ref{eq:22}) into (\ref{eq:21}), (\ref{eq:21}) then becomes:
\begin{equation*}
\begin{split}
\inf_{g(\cdot) \in \bold{P}_K} &\ g(1) \\
{\rm s.t.} \ \  & \frac{1}{k!} \frac{{\rm d}^k}{{\rm d}t^k}g(t) \ge \binom{N}{k} \sum_{l=0}^M \binom{M}{l} t^{N-k+M-l} (1-t)^{l} \\
& \quad \quad \ \ \ \ \quad \quad \ \cdot \bold{1}_{[0,1-\epsilon(k,l))} (t),\ \forall t \in [0,1],\ k \in \mathbb{N}_{0:\zeta}
\end{split}
\label{eq:23}
\end{equation*}
Denoting by $\mathcal{D}$ the feasible region of problem (\ref{eq:beta}), it has been proved in \cite{campi2018wait} that the set $\left ( \cup_{K \ge \zeta} \bold{P}_K \right ) \cap \mathcal {D}$ is dense in $\mathcal {D}$. Therefore we can conclude that:
\begin{equation*}
\mathbb{P}^{N+M} \left \{ V(\bold{x}_N^*) > \epsilon(s_N^*,r_M^*) \right \} \le \beta' \le \inf_{K} \beta_K^* = \beta^*.
\end{equation*}
This completes the proof.
\end{proof}

\begin{rem}
Theorem \ref{thm:1} enjoys wide generality in assessing a posteriori violation probabilities of scenario-based solutions, since it holds for \textit{any convex scenario program irrespective of the probability distribution $\mathbb{P}$}. Note that the randomness arises from sampling of both $N$ scenarios and $M$ validation samples, and $s_N^*$ is a random variable defined on $\Delta^N$ only, whereas $r_M^*$ is defined on the $(N+M)$-fold product probability space $\Delta^{N+M}$ since it depends on both $\bold{x}_N^*$ and validation data. Hence, the finite-sample probabilistic guarantee \eqref{eq:guarantee} only applies to a design-and-validate procedure in one-shot, and one is not allowed to repeat the validation phase over and over (with the same or a different solution) until a satisfactory performance is obtained.
\end{rem}

Next we show that with independence of $\{\epsilon(k,l)\}$ on $k$ and/or $l$ imposed, Theorem \ref{thm:1} turns out to unify many existing theoretical results, including the wait-and-judge approach in \cite{campi2018wait} as well as the generic probabilistic guarantee (\ref{eq:campi}) in \cite{campi2008exact}. First, the following result reveals itself as an obvious corollary of Theorem \ref{thm:1}, which enforces the two-indexed bounds $\{ \epsilon(k,l) \}$ to depend on $k$ only.

\begin{clr}[Theorem 1 in \cite{campi2018wait}]
\label{clr:1}
Take $\epsilon(k,l) = \epsilon(k), \forall k \in \mathbb{N}_{0:\zeta},\ l \in \mathbb{N}_{0:M}$. Under Assumptions 1 and 2, it then holds that:
\begin{equation}
\mathbb{P}^{N} \left \{ V(\bold{x}_N^*) > \epsilon(s_N^*) \right \} \le \beta^*,
\label{eq:campi_mp}
\end{equation}
where $\beta^*$ is the optimal value of the following problem:
\begin{equation*}
\begin{split}
&\ \inf_{g \in \bold{C}^{\zeta}[0,1]} g(1) \\
&\ \ \ \ \ {\rm s.t.} \ \ \ \binom{N}{k} t^{N-k} \cdot \bold{1}_{[0,1-\epsilon(k))} (t) \le \frac{1}{k!}\frac{{\rm d}^k}{{\rm d}t^k}g(t),\ k \in \mathbb{N}_{0:\zeta}
\end{split}
\end{equation*}
\end{clr}
\begin{proof}
A straightforward substitution of the identity
\begin{equation*}
\sum_{l=0}^M \binom{M}{l} (1-v)^{M-l} v^{l} = 1
\end{equation*}
into (\ref{eq:beta}) yields the proof.
\end{proof}

As an established bound in the wait-and-judge approach, Corollary \ref{clr:1} implies that only information embedded in training data can be utilized to evaluate $V(\bold{x}_N^*)$. As aforesaid, one may want to carry out Bernoulli trials on extra scenarios in conjunction with (\ref{eq:campi_mp}) to update the risk judgement. In this sense, Theorem \ref{thm:1} has wider applicability than Corollary \ref{clr:1} because the information within support constraints and empirical validation tests can be synthesized.

Based on Corollary \ref{clr:1}, one can easily cast the generic probabilistic guarantee (\ref{eq:campi}) as a corollary of Theorem \ref{thm:1} by trivially assuming constant violation probabilities.

\begin{clr}
By taking $\epsilon(k,l) \equiv \epsilon,\ \forall k \in \mathbb{N}_{0:\zeta},\ l \in \mathbb{N}_{0:M}$, and under Assumptions 1 and 2, one obtains:
\begin{equation}
\mathbb{P}^N \left \{ V(\bold{x}_N^*) > \epsilon \right \} \le B_N(\epsilon;\zeta-1).
\label{eq:campicalafi}
\end{equation}
\end{clr}
\begin{proof}
See Corollary 1 in \cite{campi2018wait}.
\end{proof}

\subsection{Derivation of A Posteriori Probabilistic Bounds}
Theorem \ref{thm:1} is shown to be a generalization of many important probabilistic bounds in scenario optimization, with particular options for $\{ \epsilon(k,l) \}$ used. Hence, the freedom in choosing $\{ \epsilon(k,l) \}$ endows the proposed scheme with capability in evaluating the solution's reliability based on all possible realizations of $s_N^*$ and $r_M^*$. One typically wishes to ensure that the event $V(\bold{x}_N^*) \le \epsilon(s_N^*,r_M^*)$ occurs with a suitably high confidence, and in order to explicitly control the probability of a wrong declaration $V(\bold{x}_N^*) > \epsilon(s_N^*,r_M^*)$, it is desirable to first assign $\beta^*$, which is typically small, and then determine $\{ \epsilon(k,l) \}$ in reverse. In principle, there are infinitely many options of $\{ \epsilon(k,l) \}$ that lead to the same confidence level $\beta^*$. For practical use, a certain choice of $\{ \epsilon(k,l) \}$ is suggested in the following theorem. It subsumes Theorem 2 in \cite{campi2018wait} as a special case with $M = 0$ (no validation tests) and specific coefficients $\{a_m \equiv 1/(N+1),\ \forall m \in \mathbb{N}_{0:N}\}$, where the induced bound $\epsilon(k)$ depending only on $s_N^*$ equals $\epsilon(k,0)$. However, the present proof machinery has a significant departure from that in \cite{campi2018wait}, which is no longer applicable to the present case where validation test results have been included and general non-negative coefficients $\{ a_m \}$ are considered. Before proceeding, we present the following preparatory lemma. \par
\begin{lemma}
\label{lem:1}
For fixed values of $m$ and $N$ $(m<N)$, $B_N(t;m)$ is a strictly decreasing function in $t \in (0,1)$.
\end{lemma}
\begin{proof}
See \cite{alamo2015randomized}.
\end{proof}

\begin{thm}
\label{thm:2}
Define the following polynomial in $t$:
\begin{equation}
\begin{split}
&\ h_{N,M}(t;k,l) \\
= &\ \beta\sum_{m=k}^N a_m \binom{m}{k} t^{m-k} - \binom{N}{k} t^{N-k} B_M(1 - t; l),
\label{eq:polynomial}
\end{split}
\end{equation}
where $\beta \in (0,1)$ is the confidence level, and coefficients $\{a_m\}$ satisfy
\begin{equation}
a_m\ge 0,\ m \in \mathbb{N}_{0:N},\ \sum_{m=\zeta}^{N-1} a_m > 0,\ \sum_{m=0}^N a_m = 1.
\label{eq:parcond}
\end{equation}

(i) $h_{N,M}(t;k,l)$ has only one solution in the interval $(0,1)$ for all $k \in \mathbb{N}_{0:\zeta}$ and $l \in \mathbb{N}_{0:M}$.

(ii) Letting the root be $t_{N,M}(k,l)$ and $\epsilon_{N,M}(k,l) = 1 - t_{N,M}(k,l)$, it then holds that:
\begin{equation}
\begin{split}
t_{N,M}(k,l) > t_{N,M}(k,l+1), \ \epsilon_{N,M}(k,l) < \epsilon_{N,M}(k,l+1), & \\
l \in \mathbb{N}_{0:M-1}. &
\end{split}
\label{eq:monoticity}
\end{equation}

(iii) The choice of above functions $\{ \epsilon_{N,M}(k,l) \}$ gives rise to the following finite-sample probabilistic guarantee:
\begin{equation*}
\mathbb{P}^{N+M} \left \{ V(\bold{x}_N^*) > \epsilon_{N,M}(s_N^*,r_M^*) \right \} \le \beta.
\end{equation*}
\end{thm}

\begin{proof} (i) First we define the following polynomial:
\begin{equation*}
\begin{split}
&\ \tilde{h}_{N,M}(t;k,l) \\
\triangleq &\  \beta\sum_{m=k}^N a_m \cdot \binom{m}{k} \cdot t^{m-N} - \binom{N}{k} B_M(1 - t; l) \\
                   =&\ h_{N,M}(t;k,l) \cdot t^{k-N}
\end{split}
\end{equation*}
Since $k\le \zeta<N$ and $\sum_{m=\zeta}^N a_m > 0$, $\sum_{m=k}^{N-1} a_m \cdot \binom{m}{k} \cdot t^{m-N}$ is a strictly decreasing function in $t \in (0,1)$. In addition, according to Lemma 1, $-B_M(1 - t; l)$ is also strictly decreasing in $t$ when $l < M$, and $B_M(1 - t; l) \equiv 1$ when $l = M$. Therefore, in both cases, $\tilde{h}_{N,M}(t;k,l)$ must be strictly decreasing in $(0,1)$. Meanwhile, notice that
\begin{gather*}
\lim_{t \to 0^+} \tilde{h}_{N,M}(t;k,l) = + \infty, \\
\begin{split}
\tilde{h}_{N,M}(1;k,l) & = \beta \sum_{m=k}^N a_m \cdot \binom{m}{k} - \binom{N}{k}B_M(0;l) \\
& \le \beta \sum_{m=0}^N a_m \cdot \binom{N}{k} - \binom{N}{k} \\
& = (\beta - 1)\binom{N}{k} < 0
\end{split}
\end{gather*}
Hence, $\tilde{h}_{N,M}(t;k,l)$ has exactly one root in $(0,1)$, and so does $h_{N,M}(t;k,l)$.

(ii) From the definition it is obvious that $B_M(1 - t; l) < B_M(1 - t; l+1)$, which implies $\tilde{h}_{N,M}(t;k,l) > \tilde{h}_{N,M}(t;k,l+1)$. Since $\tilde{h}_{N,M}(t;k,l)$ is strictly decreasing in $t \in (0,1)$, we then have \eqref{eq:monoticity}.

(iii) We first parameterize a candidate solution to (\ref{eq:beta}) as:
\begin{equation}
g(t) = \beta \sum_{m=0}^N a_m \cdot t^m.
\label{eq:candidate}
\end{equation}
Then we will show that $g(t)$ is essentially a feasible solution to (\ref{eq:beta}) with $\epsilon(k,l)$ set as $\epsilon_{N,M}(k,l)$. Its $k$-th order derivative can be computed as:
\begin{equation}
\frac{1}{k!} \frac{{\rm d}^k}{{\rm d}t^k}g(t) = \beta \sum_{m=k}^N a_m \cdot \binom{m}{k}t^{m-k},\ k \in \mathbb{N}_{1:\zeta}.
\end{equation}
Because $\epsilon_{N,M}(k,l+1) > \epsilon_{N,M}(k,l)$, the entire interval $[0,1 - \epsilon_{N,M}(k,0))$ can be split as
\begin{equation*}
\begin{split}
&\ [0,1 - \epsilon_{N,M}(k,0)) \\
= &\ \bigcup_{l=0}^M [1 - \epsilon_{N,M}(k,l+1),1 - \epsilon_{N,M}(k,l)),
\end{split}
\end{equation*}
where $\epsilon_{N,M}(k,M+1) = 1$ is taken for granted for notational clarity. Considering the case where $t$ falls within a certain interval $[1 - \epsilon_{N,M}(k,l+1),1 - \epsilon_{N,M}(k,l))$, we have:
\begin{equation*}
\begin{split}
& \frac{1}{k!} \frac{{\rm d}^k}{{\rm d}t^k}g(t) \\
&\ - \binom{N}{k} \sum_{j=0}^M \binom{M}{j} t^{N-k+M-j} (1-t)^{j} \bold{1}_{[0,1-\epsilon_{N,M}(k,j))} (t) \\
= &\ \frac{1}{k!} \frac{{\rm d}^k}{{\rm d}t^k}g(t) - \binom{N}{k} \sum_{j=0}^l \binom{M}{j} t^{N-k+M-j} (1-t)^{j} \\
= &\ \beta \sum_{m=k}^N a_m \cdot \binom{m}{k}t^{m-k} - \binom{N}{k} t^{N-k} B_M(1-t;l) \\
= &\ h_{N,M}(t;k,l) \\
> &\ 0
\end{split}
\end{equation*}
The last inequality is due to the fact that $t_{N,M}(k,l) = 1 - \epsilon_{N,M}(k,l)$ is the root of $h_{N,M}(t;k,l)$, and for $t \in [0, 1 - \epsilon_{N,M}(k,l))$, $h_{N,M}(t;k,l)>0$ holds. Therefore, if we set $\epsilon(k,l)$ in (\ref{eq:beta}) to be $\epsilon_{N,M}(k,l)$, constraints in (\ref{eq:beta}) hold for all $t \in [0,1]$. As such, $g(t)$ is feasible for problem (\ref{eq:beta}), which indicates $\beta^* \le \beta$. Finally we have:
\begin{equation*}
\mathbb{P}^{N+M} \left \{ V(\bold{x}_N^*) > \epsilon_{N,M}(s_N^*,r_M^*) \right \} \le \beta^* \le \beta.
\end{equation*}
\end{proof}

The following remarks are made in order.
\begin{rem}
The utilization of a posteriori probabilistic certificates $\{ \epsilon_{N,M}(k,l) \}$ does not require any distributional information of uncertainty. Such a distribution-free characteristic makes the proposed bounds particularly suitable for data-driven decision-making.
\end{rem}

\begin{rem}
The utilization of a posteriori probabilistic certificates $\{ \epsilon_{N,M}(k,l) \}$ should not be oriented towards a tradeoff design with a prescribed tolerance on $V(\bold{x}_N^*)$. Rather, they shall be used for a posteriori risk judgement of a scenario-based solution. Of common interest, a scenario program ${\rm SP}_N[\omega_N]$ may be constructed as a relaxation of the infinite-dimensional constraint $f(\bold{x}, \delta) \le 0,\ \forall \delta \in \Delta$, or as an approximation of the chance constraint $\mathbb{P}\{f(\bold{x}, \delta) > 0 \} \le \epsilon$ with the sample size $N$ determined based on the prior bound \eqref{eq:campicalafi}. However, the computational effort for solving ${\rm SP}_N[\omega_N]$ may be unaffordable, and one has to settle for a reduced sample size but leverages unused samples to evaluate the risk. Another situation is that after implementing the randomized solution, one may be interested in further evaluating its risk with new samples successively accumulated. In such situations, $\{ \epsilon_{N,M}(k,l) \}$ show greater promise than the wait-and-judge bounds $\{ \epsilon_{N}(k) \}$ in addressing the practical need for risk evaluation.
\end{rem}

\begin{rem}
The monotonicity property (\ref{eq:monoticity}) is consistent with our intuition that \textit{the lower the empirical violation frequency, the lower the a posteriori violation probability.} This can be also evidenced from the example in Fig. \ref{fig:grid}, where $\{\epsilon_{N,M}(k,l)\}$ with $\beta = 10^{-6}$, $\zeta = 10$, $N = 50$, $M = 30$, and $\{ a_m \equiv 1/(N+1) \}$ are portrayed. Meanwhile, it is easy to prove that $\{\epsilon_{N,M}(k,l)\}$ are increasing in $k$, which preserves the interpretation of generic wait-and-judge approach \cite{campi2018wait} that \textit{less support constraints indicate a potentially lower risk of randomized solutions}.
\end{rem}

\begin{figure}
  \centering
  \includegraphics[width=0.4\textwidth]{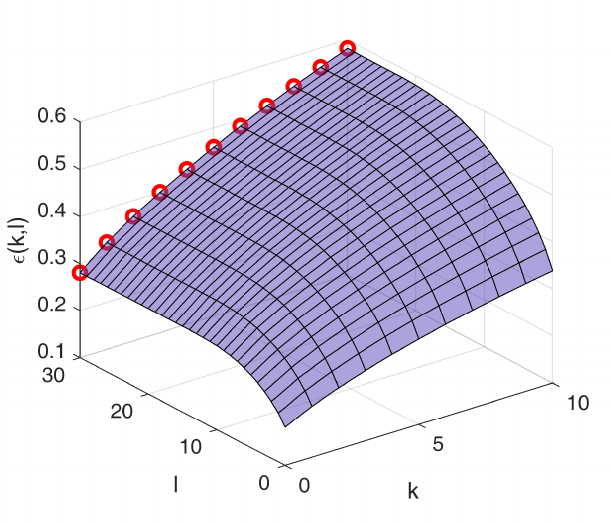}\\
  \caption{A posteriori bounds with $\beta = 10^{-6}$, $\zeta = 10$, $N = 50$, $M = 30$, and $a_m \equiv 1/(N+1)$. Note that $\{ \epsilon(\cdot,M) \}$ (marked in red circles) coincide with the a posteriori bounds in the wait-and-judge approach $\{ \epsilon(\cdot) \}$ \cite{campi2018wait}.}
  \label{fig:grid}
\end{figure}

Note that bounds $\{ \epsilon_{N,M}(k,l) \}$ and $\{ t_{N,M}(k,l) \}$ derived in Theorem \ref{thm:2} depend on $N$ and $M$. For notational simplicity, we will use $\{ \epsilon(k,l) \}$ and $\{ t(k,l) \}$ in the sequel when no confusions are caused. Meanwhile, we use $\{ \epsilon(k) \}$ to denote the a posteriori bounds in the wait-and-judge approach \cite{campi2018wait} that depends on $s_N^*$ only. Now we establish a result that reveals the value of utilizing more validation information in $\{ \epsilon_{N,M}(k,l) \}$.

\begin{clr}
\label{cor:1}
With $\{ a_m \}$ and $\beta$ fixed, it then holds that:
\begin{equation}
\begin{split}
\epsilon(k) = &\ \epsilon_{N,0}(k,0) \\
= &\ \epsilon_{N,M}(k,M) \\
> &\ \epsilon_{N,M}(k,l),\ k \in \mathbb{N}_{0:\zeta},\ l \in \mathbb{N}_{0:M-1}
\end{split}
\end{equation}
\end{clr}

\begin{proof}
The second equality can be obtained by substituting the identity $B_M(1-t;M) = 1$ into (\ref{eq:polynomial}), and the inequality stems from (\ref{eq:monoticity}).
\end{proof}

Corollary \ref{cor:1} shows that the proposed bound $\epsilon(s_N^*, r_M^*)$ is always no higher than $\epsilon(s_N^*)$ that depends on partial information within $N$ scenarios, as can be observed from Fig. \ref{fig:grid}. This is an expected outcome because having more information implies having a more accurate evaluation of a solution's risk. Moreover, even for a large $l$, $\epsilon(k,l) \le \epsilon(k)$ still holds true, which admits a sensible explanation. It has been revealed in \cite{campi2018wait,garatti2019risk} that with high confidence $r_M^* / M$ cannot be much higher than $s_N^* / N$ due to the correctness of the wait-and-judge bound $\mathbb{P}^{N} \left \{ V(\bold{x}_N^*) > \epsilon(s_N^*) \right \} \le \beta$. As a consequence, seeing an excessively large $r_M^*$ is issued with a negligible probability, so the change in $\epsilon(k,l)$ with large $l$ does not exert significant influence on the overall confidence $\beta$. Similarly, as implied by \cite{garatti2019risk}, the joint distribution of $(s_N^*,r_M^*)$ is concentrated in the region where $V(\bold{x}_N^*) = \mathbb{E}\{ r_M^* \} / M$ and $s_N^* / N$ have the same order of magnitude. This is also the region where the actual improvement of $\epsilon(k,l)$ over $\epsilon(k)$ occurs, as also confirmed by Fig. \ref{fig:grid}. In this sense, $\epsilon(k,l)$ shall be viewed as a refinement of $\epsilon(k)$ based on the new information carried by $M$ validation samples.


Next we investigate the \textit{incremental monotonicity} of generic one-sided C-P bounds $\{\eta_{M}(l)\}$, which turns out to be well inherited by $\{ \epsilon_{N,M}(k,l) \}$. This can be precisely stated as follows. Suppose that upon obtaining $\bold{x}_N^*$, $k$ support constraints have been found, and $l$ violations have occurred on $M$ validation samples $\{ \tilde{\delta}^{(j)} \}$, giving rise to the one-sided C-P bound $\eta_{M}(l)$. Afterwards, a validation sample $\tilde{\delta}^{(M+1)}$ carrying new information arrives, based on which the new probabilistic bound $\eta_{M+1}(\cdot)$ is derived. More exactly, the new bound shall be $\eta_{M+1}(l+1)$ if $f(\bold{x}^*_N,\tilde{\delta}^{(M+1)}) > 0$, and $\eta_{M+1}(l)$ if $f(\bold{x}^*_N,\tilde{\delta}^{(M+1)}) \le 0$. Before proceeding, the following lemma is presented.
\begin{lemma}
\label{lemma:app2}
For fixed values of $t$ and $m$, $B_N(t;m)$ is strictly decreasing in $N\ (N \ge m)$.
\end{lemma}
\begin{proof}
For $i \in \mathbb{N}_{1:N}$, it is obvious that $$\binom{N+1}{i} = \binom{N}{i} + \binom{N}{i-1}.$$ Therefore, we have:
\begin{equation*}
\begin{split}
&\ B_{N+1}(t;m) \\
= &\ \sum_{i=0}^m \binom{N+1}{i}t^i(1-t)^{N+1-i} \\
= &\ \underbrace{(1 - t)^{N+1} + \sum_{i=1}^m \binom{N}{i}t^i(1-t)^{N+1-i}}_{=(1-t)B_N(t;m)} \\
&\ + \underbrace{\sum_{i=1}^m \binom{N}{i - 1}t^i(1-t)^{N+1-i}}_{=t B_N(t; m - 1)} \\
= &\ (1-t)B_N(t;m) + t B_N(t; m - 1) \\
< &\ (1-t)B_N(t;m) + t B_N(t;m) \\
= &\ B_N(t;m)
\end{split}
\end{equation*}
This completes the proof of Lemma \ref{lemma:app2}.
\end{proof}

Then the following relations hold between $\{ \eta_M(\cdot) \}$ and $\{ \eta_{M+1}(\cdot) \}$.
\begin{thm}
\label{thm:33}
It holds that:
\begin{gather}
\eta_{M}(l) > \eta_{M+1}(l),\ l \in \mathbb{N}_{0:M}, \\
\eta_{M+1}(l+1) > \eta_{M}(l),\ l \in \mathbb{N}_{0:M-1}, \\
\eta_{M+1}(M+1) = \eta_{M}(M) = 1. \label{eq:26}
\end{gather}
\end{thm}
\begin{proof}
(\ref{eq:26}) immediately follows from (\ref{eq:pmk}). Because $\eta_{M}(l)$ is the root of $B_M(\eta,l) = \beta^*$ in $(0,1)$, and $B_M(\eta,l)$ is strictly decreasing in $\eta$, it remains to show that:
\begin{gather}
B_M(\eta; l) > B_{M + 1}(\eta; l),\ l \in \mathbb{N}_{0:M}, \label{eq:29} \\
B_{M+1}(\eta; l+1) > B_M(\eta; l),\ l \in \mathbb{N}_{0:M-1}. \label{eq:30}
\end{gather}
A straightforward usage of Lemma \ref{lemma:app2} results in (\ref{eq:29}), along with the inequality $B_{M+1}(1 - \eta; M-l-1) < B_M(1 - \eta; M-l-1),\ l \in \mathbb{N}_{0:M-1}$. Hence, $1 - B_{M + 1}(1 - \eta; M-l-1) > 1 - B_M(1 - \eta;M-l-1)$ holds, which amounts to (\ref{eq:30}).
\end{proof}

In view of Theorem \ref{thm:33}, rational refinements of $\epsilon(s_N^*, r_M^*)$ can be made once more validation samples are successively accumulated. If constraint violation is witnessed on a new validation sample, a higher risk level of $\bold{x}_N^*$ is obtained, which is precisely stated as $\eta_{M+1}(l+1) > \eta_{M}(l)$. Conversely, if $\bold{x}_N^*$ agrees with the new validation sample, the risk level should be discounted, i.e. $\eta_{M+1}(l) < \eta_{M}(l)$. Interestingly, such properties are also possessed by $\{ \epsilon_{N,M}(k,l) \}$.

\begin{thm}
\label{thm:4}
With values of $\{ a_m \}$ and $\beta$ fixed, the roots of $\{ g_{N,M}(t;k,l) \}$ satisfy
\begin{gather*}
t_{N,M}(k,l) < t_{N,M+1}(k,l),\ l \in \mathbb{N}_{0:M}, \\
t_{N,M+1}(k,l+1) < t_{N,M}(k,l),\ l \in \mathbb{N}_{0:M-1}, \\
t_{N,M+1}(k,M+1) = t_{N,M}(k,M),
\end{gather*}
thereby indicating the following relations:
\begin{gather}
\epsilon_{N,M}(k,l) > \epsilon_{N,M+1}(k,l),\ l \in \mathbb{N}_{0:M}, \label{eq:thm31} \\
\epsilon_{N,M+1}(k,l+1) > \epsilon_{N,M}(k,l),\ l \in \mathbb{N}_{0:M-1}, \label{eq:thm32} \\
\epsilon_{N,M+1}(k,M+1) = \epsilon_{N,M}(k,M). \label{eq:thm33}
\end{gather}
\end{thm}

The proof is similar to that for Theorem \ref{thm:33}, and is hence omitted for brevity.

Before closing this subsection, we deal with the computational issue. Although $\{\epsilon_{N,M}(k,l)\}$ have no analytic expressions, one only needs to perform bisection to numerically search the root of $\tilde{h}_{N,M}(t;k,l)$ in $(0,1)$ to compute $\{\epsilon_{N,M}(k,l)\}$, as implied by Theorem \ref{thm:2}(i). The pseudo code for calculating $\{\epsilon_{N,M}(k,l)\}$ for all $k$ and $l$ is outlined in Algorithm 1, which can be readily implemented with numerical softwares. In order to reduce the search space, the monotonicity property (\ref{eq:monoticity}) is utilized, and the enumeration of $l$ is made from $M$ to $0$. For a particular choice of $k$ and $l$, one only needs to execute the inner-loop in Algorithm 1 with $LB = 0$ and $UB = 1$.

\begin{table}[htbp]
\centering
\renewcommand\arraystretch{1.2}
\begin{tabular}{p{0.46\textwidth}}
\hline
\textbf{Algorithm 1} Bisection Method for Computing A Posteriori Bounds \\
\hline
\textbf{Inputs}: Integers $\zeta$, $N$ and $M$, confidence level $\beta$, and coefficients $\{ a_m \}$. \\ 1: \ \textbf{Initialization}: Set numerical accuracy $\varepsilon_{\rm tol}$, and $t(k,M+1) \equiv 0,\ \forall k$. \\
2: \ \textbf{for} $k = 0, 1, \cdots, \zeta$ \\
3: \ \qquad \textbf{for} $l = M,M-1 \cdots, 0$ \\
4: \ \qquad \qquad $LB = t(k,l+1)$, $UB = 1$; \\
5: \ \qquad \qquad \textbf{while} $UB - LB \ge \varepsilon_{\rm tol}$ \textbf{do} \\
6: \ \qquad \qquad \qquad $t^{\rm new} = (UB + LB) / 2$; \\
7: \ \qquad \qquad \qquad \textbf{if} $\tilde{h}_{N,M}(t^{\rm new};k,l) \ge 0$ \\
8: \ \qquad \qquad \qquad \qquad $LB = t^{\rm new}$; \\
9: \ \qquad \qquad \qquad \textbf{else} \\
10: \qquad \qquad \qquad \qquad $UB = t^{\rm new}$; \\
11: \quad \ \ \qquad \qquad \textbf{end} \\
12: \ \ \quad \qquad \textbf{end} \\
13: \ \ \quad \qquad $t(k,l) = (UB + LB) / 2,\ \epsilon(k,l) = 1 - t(k,l)$. \\
14: \qquad \textbf{end} \\
15: \textbf{end} \\
16: \textbf{Return} $\{ \epsilon_{N,M}(k,l) \}$. \\
\hline
\end{tabular}
\label{tab:algorithm}
\end{table}

\subsection{Fundamental limits and a refinement procedure}
Above derivations imply that lower values of $\{ \epsilon(k,l) \}$ give more informative certificates on the solution's reliability. Hence a natural question comes up: given the confidence $\beta^*$, what is the best possible value of $\epsilon(k,l)$ satisfying \eqref{eq:guarantee} in Theorem \ref{thm:1}? Next we establish lower limits for $\{ \epsilon(k,l) \}$, where $\epsilon(k,l)$ is restricted to be strictly increasing in $l$, which is in agreement with the intuition that seeing more violations leads to a discounted reliability level.

\begin{thm}
\label{thm:5}
Consider all admissible functions $\{ \epsilon(k,l) \}$ satisfying (i) the probabilistic guarantee \eqref{eq:guarantee} in Theorem \ref{thm:1} with a prescribed $\beta^*$ and (ii) the monotonicity property $\epsilon(k,l) < \epsilon(k,l+1), ~l \in \mathbb{N}_{0:M-1}$. Then for $k \ge 1$ a lower limit $\underline{\epsilon}(k,l)$ of $\epsilon(k,l)$ is given by the root of the following equation in $(0,1)$:
\begin{equation}
\sum_{j=0}^l z_j B_{N + M}(\epsilon; k + j - 1) = \beta^*,
\label{eq:36}
\end{equation}
where
\begin{equation*}
z_j = \frac{\binom{N}{k} \binom{M}{j}}{\binom{N+M}{k+j}} \cdot \frac{k}{k + j}.
\end{equation*}
And, $\underline{\epsilon}(0,l) = 0,\ \forall l \in \mathbb{N}_{0:M}$.
\end{thm}

\begin{proof}
We show by contradiction that $\epsilon(k,l) < \underline{\epsilon}(k,l)$ is impossible for $k \ge 1$. Assume that there exist admissible functions $\{ \epsilon(k,l) \}$ satisfying (i) and (ii), and there exist $k'$ and $l'$ such that $\epsilon(k',l') < \underline{\epsilon}(k',l')$. Then we consider a fully-supported problem ${\rm SP}'_N[\omega_N]$ where $s_N^* = k' \le \zeta$ occurs w.p.1. When $k' \ge 1$, the probability density function of $V(\bold{x}_N^*)$ in ${\rm SP}'_N[\omega_N]$ is explicitly expressed as \cite{campi2008exact}:
\begin{equation}
p(\alpha) = \binom{N}{k'} \alpha^{k'-1} \cdot k' \cdot (1-\alpha)^{N-k'}.
\end{equation}
Hence, for this specific fully-supported problem ${\rm SP}'_N[\omega_N]$, it follows that:
\begin{equation*}
\begin{split}
&\ \mathbb{P}^{N+M} \left \{ V(\bold{x}_N^*) > \epsilon(s_N^*,r_M^*) \right \} \\
= &\ \sum_{j=0}^M \mathbb{P}^{N+M} \left \{ V(\bold{x}_N^*) > \epsilon(k',j) \land r_M^* = j \right \}. \\
= &\ \sum_{j=0}^M \int_{\epsilon(k',j)}^1 \binom{M}{j} p(\alpha) \cdot \alpha^j \cdot (1-\alpha)^{M-j} {\rm d} \alpha \\
= &\ \sum_{j=0}^M \int_{\epsilon(k',j)}^1 \binom{N}{k'}\binom{M}{j}k' \alpha^{k'+j-1} \cdot  (1-\alpha)^{N+M-k'-j} {\rm d} \alpha \\
= &\ \mathrm{[Integration\ by\ parts]} \\
= &\ \sum_{j=0}^M z_j \sum_{i=0}^{k'+j-1} \binom{N+M}{i} \epsilon(k',j)^i [1 - \epsilon(k',j)]^{N+M-i} \\
= &\ \sum_{j=0}^M z_j B_{N + M}(\epsilon(k',j); k' + j - 1)
\end{split}
\end{equation*}
Recall that $\epsilon(k',l) < \underline{\epsilon}(k',l)$ has been assumed, and $\{ \epsilon(k,l) \}$ satisfies the monotonicity property $\epsilon(k,l) < \epsilon(k,l+1), ~l \in \mathbb{N}_{0:M-1}$. It immediately follows that $\epsilon(k',j) < \underline{\epsilon}(k',l),\ \forall j \in \mathbb{N}_{0:l}$, and one obtains:
\begin{equation}
\begin{split}
&\ \mathbb{P}^{N+M} \left \{ V(\bold{x}_N^*) > \epsilon(s_N^*,r_M^*) \right \} \\
= &\ \sum_{j=0}^M z_j B_{N + M}(\epsilon(k',j); k' + j - 1) \\
\ge &\ \sum_{j=0}^l z_j B_{N + M}(\epsilon(k',j); k' + j - 1) \\
> &\ \sum_{j=0}^l z_j B_{N + M}(\underline{\epsilon}(k',l); k' + j - 1) \\
= &\ \beta^*
\end{split}
\label{eq:366}
\end{equation}
where the last inequality stems from Lemma 1 that $B_{M}(t;m)$ is strictly decreasing in $t$, and the last equality is due to \eqref{eq:36}. However, \eqref{eq:366} indicates that ${\rm SP}'_N[\omega_N]$ contradicts the distribution-free nature of $\{ \epsilon(k,l) \}$ in (i), i.e., the probabilistic guarantee \eqref{eq:guarantee} with a given $\beta^*$.
\end{proof}

Fig. \ref{fig:LB} depicts a posteriori bounds $\{ \epsilon(k,l) \}$ under $\{a_m \equiv 1 / (N+1)\}$ and the fundamental lower limits $\{ \underline{\epsilon}(k,l) \}$ ($\beta = 10^{-6}$, $\zeta = 8$, $N = 100$, $M = 5$), where two surfaces are not far away from each other except for $k=0$. Hence, we pay a small price for distribution-free certificates $\{ \epsilon(k,l) \}$ since even knowing $k \ge 1$ beforehand does not provide too much improvement in risk evaluation. Indeed, this well inherits the merits of the wait-and-judge scenario optimization \cite{campi2018wait,garatti2019risk}, which can be interpreted as that the realization of $s_N^*$ is rather informative for a posteriori risk evaluation \cite{garatti2019risk}.

\begin{figure}
  \centering
  \includegraphics[width=0.4\textwidth]{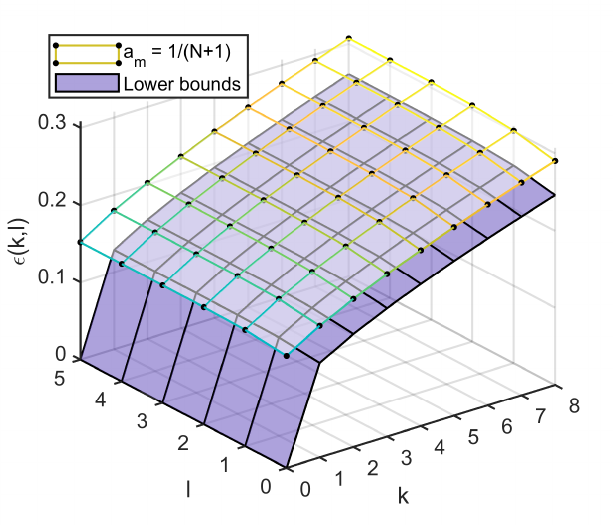}\\
  \caption{Comparison between a posteriori bounds $\{ \epsilon(k,l) \}$ derived by $\{a_m \equiv 1 / (N+1)\}$ and lower bounds $\{ \underline{\epsilon}(k,l) \}$ ($\beta = 10^{-6}$, $\zeta = 8$, $N = 100$, $M = 5$).}
  \label{fig:LB}
\end{figure}

Obviously, $\{ \epsilon(k,l) \}$ are decided by coefficients $\{a_m\}$, and different choice of $\{a_m\}$ lead to different probabilistic certificates. A natural question is that whether $\{ \epsilon(k,l) \}$ can be improved to approach the fundamental limits $\{ \underline{\epsilon}(k,l) \}$ by flexibly adjusting $\{a_m\}$. This can be formally cast as the following multi-objective optimization problem:
\begin{equation}
\begin{split}
\max\ &\ t(k,l),\ k\in\mathbb{N}_{0:\zeta},\ l\in\mathbb{N}_{0:M} \\
{\rm s.t.}\ &\ 0< t(k,l) <1,\ k\in\mathbb{N}_{0:\zeta},\ l\in\mathbb{N}_{0:M} \\
                 &\ \ \ \beta\sum_{m=k}^N a_m \binom{m}{k} t(k,l)^{m-N} \\
                 & = \binom{N}{k} B_M(1 - t(k,l); l),\ k\in\mathbb{N}_{0:\zeta},\ l\in\mathbb{N}_{0:M} \\
                 &\ \sum_{m=\zeta}^{N-1} a_m \ge \tau,\ \sum_{m=0}^N a_m = 1,\ a_m \ge 0,\ m \in\mathbb{N}_{0:N}
\end{split}
\label{eq:multiobj}
\end{equation}
where $\tau > 0$ is a small positive number to ensure (\ref{eq:parcond}). Clearly, (\ref{eq:multiobj}) is non-convex and is hence difficult to solve. Next we seek to devise a tailored algorithm. Suppose that a certain choice of coefficients $\{a_m\}$ is feasible for (\ref{eq:multiobj}). The following theorem provides a sufficient and necessary condition for making improvements over the current $\{a_m\}$.

\begin{thm}
\label{thm:7}
Given $N$, $M$ and $\beta$. Assume that $\{a_m\}$ and $\{a'_m\}$ satisfying (\ref{eq:parcond}) are different coefficients of polynomials $\tilde{h}_{N,M}(t;k,l)$ and $\tilde{h}'_{N,M}(t;k,l)$, whose roots in $(0,1)$ are $\{t(k,l)\}$ and $\{t'(k,l)\}$, respectively. Then $t'(k,l) \ge (>) t(k,l)$ holds if and only if
\begin{equation}
\beta\sum_{m=k}^N a'_m \binom{m}{k} t(k,l)^{m-N} \ge (>) \binom{N}{k} B_M(1 - t(k,l);l).
\label{eq:42}
\end{equation}
\end{thm}

\begin{proof}
It is known that $\tilde{h}'_{N,M}(t;k,l)$, whose root is $t'(k,l)$, is strictly decreasing in $(0,1)$. Therefore, when $t'(k,l) \ge (>) t(k,l)$, it holds that $\tilde{h}'_{N,M}(t(k,l);k,l) \ge (>) 0$, which leads to (\ref{eq:42}). The reverse is also true, which completes the proof.
\end{proof}

To derive refined coefficients $\{ a'_m \}$ based on present ones $\{ a_m \}$, it suffices to ensure (\ref{eq:42}) for all $k$ and $l$ and lift the LHS of (\ref{eq:42}) as much as possible. This can be achieved, for instance, by simply solving the following linear program (LP):
\begin{equation}
\begin{split}
\max_{\{ a_m' \}} &\ \sum_{k=0}^{\zeta} \sum_{l=0}^M \sum_{m=k}^N a_m' \binom{m}{k} t(k,l)^{m-N} \\
{\rm s.t.} &\ \ \quad \beta\sum_{m=k}^N a'_m \binom{m}{k} t(k,l)^{m-N} \\
&\ \ge \binom{N}{k} B_M(1 - t(k,l);l),\ \forall k \in \mathbb{N}_{0:\zeta}, \ l \in \mathbb{N}_{0:M} \\
&\ \sum_{m=\zeta}^{N-1} a_m' \ge \tau,\ \sum_{m=0}^N a_m' = 1,\ a_m' \ge 0,\ m \in\mathbb{N}_{0:N}
\end{split}
\label{eq:LP}
\end{equation}

By successively solving LP (\ref{eq:LP}), coefficients $\{ a_m \}$ and the associated $\{t(k,l)\}$ can be refined. The implementation details are summarized in Algorithm 2. It turns out that the resultant roots $\{t(k,l)\}$ and a posteriori bounds $\{ \epsilon(k,l) \}$ bear Pareto optimality.

\begin{table}[htbp]
\centering
\renewcommand\arraystretch{1.2}
\begin{tabular}{p{0.46\textwidth}}
\hline
\textbf{Algorithm 2} Refinement Procedure of A Posteriori Probability Bounds \\
\hline
\textbf{Inputs}: Integers $\zeta$, $N$ and $M$, confidence level $\beta$. \\
1: \textbf{Initialization}: Set coefficients $\{ a_m^{({\rm ite})} \}$ and iteration counter ${\rm ite} = 1$. \\
2: \textbf{Do until convergence} \\
3: \qquad - Compute $\{t^{({\rm ite})}(k,l)\}$ based on $\{a_m^{({\rm ite})} \}$. \\
4: \qquad - Solve (\ref{eq:LP}) with $\{ t^{({\rm ite})}(k,l) \}$ and obtain optimum $\{ a_m^{({\rm ite}+1)} \}$. \\
5: \qquad - Update ${\rm ite} = {\rm ite} + 1$. \\
6: \textbf{end} \\
7: \textbf{Return} $\epsilon(k,l) = 1 - t^{({\rm ite})}(k,l),\ \forall k,\ \forall l$. \\
\hline
\end{tabular}
\label{tab:algorithm2}
\end{table}

\begin{thm}
\label{thm:8}
In Algorithm 2, the sequence $\{t^{({\rm ite})}(k,l)\}$ converges, i.e.,
\begin{equation}
\lim_{{\rm ite} \to +\infty} t^{({\rm ite})}(k,l) = t^*(k,l) < + \infty,\ \forall k,\ l,
\end{equation}
and $\{t^*(k,l)\}$ lie on the Pareto optimal boundary of (\ref{eq:multiobj}). That is, there exist no coefficients $\{\tilde{a}^*_m\}$ that give $\{ \tilde{t}^*(k,l) \}$ strictly dominating $\{ t^*(k,l) \}$.
\end{thm}

\begin{proof}
Following the same notations in Theorem \ref{thm:2}(i), we consider the polynomial $\tilde{h}_{N,M}(t;k,l)$ parameterized by $\{ a_m^{({\rm ite}+1)} \}$. On one hand, $\tilde{h}_{N,M}(t^{({\rm ite} + 1)}(k,l);k,l) = 0$. On the other hand, $\tilde{h}_{N,M}(t^{({\rm ite})}(k,l);k,l) \ge 0$ due to \eqref{eq:LP}. As a consequence, $t^{({\rm ite})}(k,l) \le t^{({\rm ite} + 1)}(k,l)$ holds in each iteration due to the monotonicity of $\tilde{h}_{N,M}(t;k,l)$ established in Theorem \ref{thm:2}(i). Meanwhile, since $t^{({\rm ite})}(k,l)$ is upper-bounded from $1$, it must converge to a finite limit $t^*(k,l)$. However, the sequence $\{ a_m^{({\rm ite})} \}$ does not necessarily converge. Because $0 \le a_m^{({\rm ite})} \le 1$, there is a convergent subsequence $\{ a_m^{({\rm ite}_i)} \}$ with indices $1 \le {\rm ite}_1 < {\rm ite}_2 < \cdots$ according to the Bolzano-Weierstrass theorem, i.e.,
\begin{equation}
\lim_{{i} \to +\infty} a_m^{({\rm ite}_i)} = a_m^* < +\infty,\ \forall m.
\end{equation}

Because $\tilde{h}_{N,M}(t;k,l)$ is continuous, by a continuity argument it is not difficult to show that $\{ a_m^* \}$ is an optimal solution to (\ref{eq:LP}) with $\{t(k,l) = t^*(k,l)\}$, and $\{t^*(k,l)\}$ are exactly the roots of $\tilde{h}_{N,M}(t;k,l)$ induced by $\{ a_m^* \}$. Then we proceed by contradiction. Suppose that there exist $\{\tilde{a}^*_m\}$ that give $\{ \tilde{t}^*(k,l) \}$ strictly dominating $\{ t^*(k,l) \}$. That is, $\tilde{t}^*(k,l) \ge t^*(k,l),~\forall k,l$ and there exists $(k',l')$ such that $\tilde{t}^*(k',l') > t^*(k',l')$. Then it immediately follows from Theorem \ref{thm:7} that $\forall k,l$:
\begin{gather*}
\begin{split}
\beta\sum_{m=k}^N \tilde{a}^*_m \binom{m}{k} t^*(k,l)^{m-N} \ge & \binom{N}{k} B_M(1 - t^*(k,l);l) \\
= & \beta\sum_{m=k}^N a^*_m \binom{m}{k} t^*(k,l)^{m-N},
\end{split}
\end{gather*}
and for the particular indices $(k',l')$,
\begin{gather*}
\begin{split}
\beta\sum_{m=k'}^N \tilde{a}^*_m \binom{m}{k'} t^*(k',l')^{m-N} > & \binom{N}{k'} B_M(1 - t^*(k',l');l') \\
= & \beta\sum_{m=k'}^N a^*_m \binom{m}{k'} t^*(k',l')^{m-N}.
\end{split}
\end{gather*}
Consequently, $\{\tilde{a}^*_m\}$ is not only feasible for (\ref{eq:LP}) with $\{t(k,l) = t^*(k,l)\}$, but also leads to an objective value strictly higher than $\{ a_m^* \}$ does. This contradicts the optimality of $\{ a_m^* \}$.
\end{proof}

To obtain different Pareto optimal certificates, one could use different initializations of $\{ a_m \}$. Fig. \ref{fig:bound} profiles $\{ \epsilon(k,l) \}$ derived with $\{a_m \equiv 1 / (N+1)\}$ and those after refinement with $\beta = 10^{-6}$, $\zeta = 8$, $N = 100$, $M = 5$. It can be seen that values of $\{ \epsilon(k,l) \}$ for all $k$ and $l$ become lower after refinement. While doing so one shall keep in mind that no matter how much improvement can be made over $\{ \epsilon(k,l) \}$, one cannot go below the fundamental limits $\{ \underline{\epsilon}(k,l) \}$ in Fig. \ref{fig:LB}.


\begin{figure}
  \centering
  \includegraphics[width=0.4\textwidth]{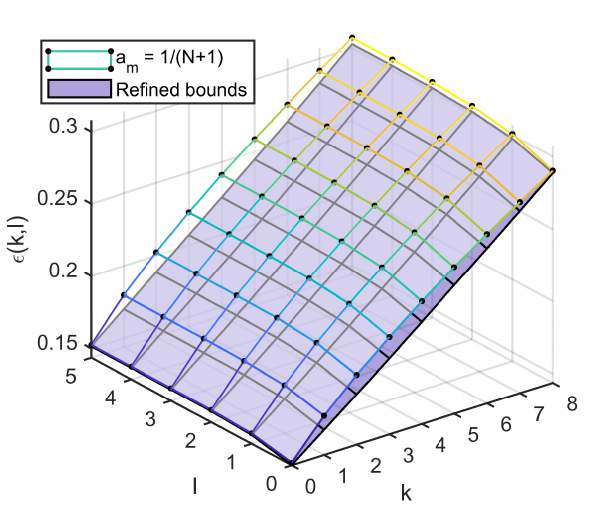}\\
  \caption{Profiles of a posteriori bounds $\{ \epsilon(k,l) \}$ derived by $\{a_m \equiv 1 / (N+1)\}$ and the refined bounds ($\beta = 10^{-6}$, $\zeta = 8$, $N = 100$, $M = 5$).}
  \label{fig:bound}
\end{figure}

\section{Case Studies on LQR Design of Aircraft Lateral Motion}
\label{sec:5}
\subsection{Problem Setup}
In this section, we adopt a simulation example of controller design of aircraft lateral motion from \cite{calafiore2006probabilistic}, which has been widely used as a testbed for probabilistic control design \cite{tempo2012randomized,levine2018probabilistic}, to demonstrate the developed theory. The system is described by the following state-space model:
\begin{equation*}
\dot{x}(t) = A(\delta)x(t) + B(\delta)u(t),
\end{equation*}
where system states $x(t) = [x_1(t)\ x_2(t)\ x_3(t)\ x_4(t)]^{\rm T}$ are the blank angle and its derivative, the sideslip angle, and the yaw rate, respectively. Control inputs $u(t) = [u_1(t)\ u_2(t)]^{\rm T}$ include the rudder deflection and the aileron deflection. System matrices $A(\delta)$ and $B(\delta)$ are influenced by uncertain parameters $\delta$, as given by:
\begin{gather*}
A(\delta) = \left [
\begin{array}{c c c c}
0 & 1 & 0 & 0 \\
0 & L_p & L_{\beta} & L_r \\
g/V & 0 & Y_{\beta} & -1 \\
N_{\dot{\beta}} & N_p & N_{\beta} + N_{\dot{\beta}}Y_{\beta} & N_r - N_{\dot{\beta}} \\
\end{array}
\right ], \\
B(\delta) = \left [
\begin{array}{c c}
0 & 0 \\
0 & L_{\delta_a} \\
Y_{\delta_r} & 0 \\
N_{\delta_r} + N_{\dot{\beta}}Y_{\delta_r} & N_{\delta_a} \\
\end{array}
\right ],
\end{gather*}
\begin{equation*}
\delta = [L_p\ L_{\beta}\ L_r\ g/V\ Y_{\beta}\ N_{\dot{\beta}}\ N_p\ N_r\ L_{\delta_a}\ Y_{\delta_r}\ N_{\delta_r}\ N_{\delta_a}]^{\rm T}.
\end{equation*}
It is assumed that a $10\%$ perturbation of $\delta$ is added around its nominal value $\bar{\delta}$, which is uniformly distributed.

The control goal is to design a state feedback controller $u=Kx$ such that real parts of eigenvalues of the closed-loop system are smaller than $-\gamma < 0$. That is, the system is stabilized with a desired decay rate $\gamma > 0$. In this case we choose $\gamma = 0.5$. A possible formulation of the controller design problem is given by the following semi-definite program (SDP), which involves an infinite number of linear matrix inequalities (LMIs):
\begin{equation}
\begin{split}
 \min_{P \succ 0,Y}  &\ {\rm Tr} \{ P \} \\
{\rm s.t.}\ &\ P \succeq \theta I \\
           &\ A(\delta)P + PA(\delta)^{\rm T} + B(\delta)Y + Y^{\rm T}B(\delta) + 2\gamma P \preceq 0, \\
           & \quad \quad \quad \quad \quad \quad \quad \quad \quad \quad \quad \quad \quad \quad \quad \quad \quad \quad \quad \forall \delta \in \Delta
\end{split}
\label{eq:LQRdesign}
\end{equation}
where $P \in \mathbb{R}^{4 \times 4}$ and $Y \in \mathbb{R}^{2 \times 4}$. $\theta = 0.01$ is a small positive number to ensure the positive definiteness of $P$. The control gain $K$ can be finally computed as $K = YP^{-1}$.

\begin{table*}
\caption{Results of Different A Posteriori Bounds from 2000 Monte Carlo Runs ($\beta = 10^{-6}$)}
\centering
\begin{tabular}{c c | c c | c c | c c}
\hline
\multirow{2}{*}{$N$} & \multirow{2}{*}{$M$} & \multicolumn{2}{c|}{$\epsilon(s_N^*)$} & \multicolumn{2}{c|}{$\eta(r_M^*)$} & \multicolumn{2}{c}{$\epsilon(s_N^*,r_M^*)$} \\
 &  & $\mu_{\rm GAP}$ & $\sigma_{\rm GAP}$ & $\mu_{\rm GAP}$ & $\sigma_{\rm GAP}$ & $\mu_{\rm GAP}$ & $\sigma_{\rm GAP}$ \\
\hline
\multirow{4}{*}{$100$} & $100$ & \multirow{4}{*}{$18.18\%$} & \multirow{4}{*}{$2.74\%$} & $16.92\%$ & $3.63\%$ & $10.25\%$ & $1.83\%$ \\
                       & $200$ &                            &                           & $10.37\%$ & $2.49\%$ & $7.69\%$  & $1.50\%$ \\
                       & $500$ &                            &                           &  $5.67\%$ & $1.49\%$ & $4.89\%$  & $1.07\%$ \\
                       & $1000$&                            &                           &  $3.67\%$ & $1.03\%$ & $3.36\%$  & $0.80\%$ \\
\hline
\multirow{4}{*}{$200$} & $100$ & \multirow{4}{*}{$9.82\%$} & \multirow{4}{*}{$1.50\%$} & $15.32\%$ & $2.89\%$ & $6.74\%$ & $1.14\%$ \\
                       & $200$ &                            &                           &  $8.99\%$ & $2.05\%$ & $5.38\%$ & $0.98\%$ \\
                       & $500$ &                            &                           &  $4.61\%$ & $1.24\%$ & $3.58\%$ & $0.77\%$ \\
                       & $1000$&                            &                           &  $2.90\%$ & $0.83\%$ & $2.51\%$ & $0.60\%$ \\
\hline
\multirow{4}{*}{$500$} & $100$ & \multirow{4}{*}{$4.13\%$} & \multirow{4}{*}{$0.62\%$} & $14.00\%$ & $2.06\%$ & $3.40\%$ & $0.53\%$ \\
                       & $200$ &                            &                           &  $7.78\%$ & $1.46\%$ & $2.96\%$ & $0.48\%$ \\
                       & $500$ &                            &                           &  $3.72\%$ & $0.86\%$ & $2.22\%$ & $0.40\%$ \\
                       & $1000$&                            &                           &  $2.23\%$ & $0.57\%$ & $1.65\%$ & $0.34\%$ \\
\hline
\end{tabular}
\label{tab:MonteCarlo}
\end{table*}

Assume that there is no complete knowledge about $\Delta$ and only some samples of $\delta$ can be collected. To obtain an approximate solution to (\ref{eq:LQRdesign}), we turn to the following scenario program with a total of $d = 18$ free decision variables in matrices $P$ and $Y$:
\begin{equation}
\begin{split}
 \min_{P \succ 0,Y}  &\ {\rm Tr} \{ P \} \\
{\rm s.t.}\ &\ P \succeq \theta I \\
           &\ A(\delta^{(i)})P + PA(\delta^{(i)})^{\rm T} + B(\delta^{(i)})Y + Y^{\rm T}B(\delta^{(i)}) \\
           & \qquad \qquad \qquad \qquad \qquad + 2\gamma P \preceq 0,\ i \in \mathbb{N}_{0:N}
\end{split}
\label{eq:SDP}
\end{equation}
where $\{ \delta^{(i)} \}_{i=1}^N$ are randomly collected scenarios of uncertain parameters. To solve the large-scale SDP (\ref{eq:SDP}), we use {\tt cvx} package in MATLAB equipped with the MOSEK solver \cite{grant2008cvx}.

\subsection{Results and Discussions}
In the simulation phase, we randomly generate $N = 500$ independent scenarios $\{ \delta^{(i)} \}$ for solving (\ref{eq:SDP}), and $M = 500$ validation scenarios for empirically estimating the violation frequency of the LMI in (\ref{eq:LQRdesign}). The confidence is set as $\beta = 10^{-6}$ (practical certainty). Since there is no knowledge about a tighter bound of Helly's dimension, we set $\zeta = d = 18$. The generic a priori bound (\ref{eq:campi}) yields $\epsilon = 0.0889$. The attractiveness of a posteriori bound lies in that, upon seeing $\bold{x}_N^*$, $s_N^*$ and $r_M^*$, the certificate of $V(\bold{x}_N^*)$ can be refined. For example, in a particular simulation run, $s_N^* = 3$ support constraints and $r^*_M = 2$ times of violations have been revealed. Accordingly, we use $a_m \equiv 1/(N+1)$ to derive $\epsilon(s_N^*,r^*_M) = 0.0268$, indicating that with confidence $99.9999\%$ the violation probability $V(\bold{x}_N^*)$ is no more than $2.68\%$, which is much tighter than the a priori bound. In contrast, $\epsilon(s_N^*) = 0.0486$ with support constraints information used only \cite{campi2018wait}. If the one-sided C-P bound (\ref{eq:onesided}) for Bernoulli trials is adopted, we obtain $\eta_M(r_M^*) = 0.0376$. Hence in this instance, by synthesizing comprehensive information from both decisive support constraints and validation tests, a tighter certificate on the violation probability $V(\bold{x}_N^*)$ can be obtained.

Next, we carry out 2000 Monte Carlo simulation runs for $N=100,200,500$ and $M=100,200,500,1000$ with $\beta = 10^{-6}$. In each run, after deriving the solution $\bold{x}^*_N$, we obtain a high-fidelity estimate $\hat{V}(\bold{x}^*_N)$ of its true violation probability $V(\bold{x}^*_N)$ with $10^5$ additional Monte Carlo samples. Then the gap between a certain bound with $\hat{V}(\bold{x}^*_N)$ can be calculated as a performance quantification, which, for instance, is ${\rm GAP} = \epsilon(s_N^*,r_M^*) - \hat{V}(\bold{x}^*_N)$ for the proposed bound. By summarizing results in 2000 runs, the mean value $\mu_{\rm GAP}$ and the standard deviation $\sigma_{\rm GAP}$ of the gap are further calculated and summarized in Table \ref{tab:MonteCarlo}.

Many interesting observations can be attained from Table \ref{tab:MonteCarlo}. First, in all cases the proposed bound $\epsilon(s_N^*,r_M^*)$ yields the smallest mean and standard deviation of the gap, indicating its significantly reduced conservatism. This is because both $\epsilon(s_N^*)$ and $\eta_M(r_M^*)$ only rely on partial information. With fixed $N$ and small $M$, $\eta_M(r_M^*)$ tends to be quite conservative, while the proposed bounds provides the tightest certificates by using comprehensive information. With $M$ increasing, both $\eta_M(r_M^*)$ and $\epsilon(s_N^*,r_M^*)$ get lower; when $M$ is sufficiently large, the difference between $\epsilon(s_N^*,r_M^*)$ and $\eta_M(r_M^*)$ vanishes, because the effect of validation tests becomes increasingly dominant. Note that in this case the wait-and-judge approach gives a constant bound $\epsilon(s_N^*)$ without leveraging validation information. On the other hand, with $N$ increasing, the conservatism of both $\epsilon(s_N^*)$ and $\epsilon(s_N^*,r_M^*)$ is reduced, since a larger $N$ leads to improved robustness of randomized solutions. This can also be evidenced from the fact that, under the same $M$, the difference between $\epsilon(s_N^*,r_M^*)$ and $\epsilon(s_N^*)$ declines with $N$ increasing, while the difference between $\eta_M(r_M^*)$ and $\epsilon(s_N^*,r_M^*)$ becomes more pronounced. Henceforth, the proposed bound $\epsilon(s_N^*,r_M^*)$ features a sophisticated integration of information from both support constraints and validation tests. When $N$ is small, $\epsilon(s_N^*,r_M^*)$ tends to depend more on validation test, while when $M$ is small, information carried by $s_N^*$ becomes dominant. This yields an explanation for empirically better performances of $\epsilon(s_N^*,r_M^*)$ under various choices of $N$ and $M$.

Table \ref{tab:freq} reports frequencies of $s_N^*$ and statistics of $r_M^*$ with $M = 1000$ during 2000 Monte Carlo runs. The realization of $s_N^*$ is always between $2$ and $8$, which is much smaller than $d = 18$. Hence, the information carried by $s_N^*$ helps give a non-conservative certificate due to the non-fully-supportedness, which is not rare in practical situations \cite{care2014fast,welsh2011robust}. In addition, as analyzed in Section III, the actual improvement in $\epsilon(s_N^*, r_M^*)$ occurs primarily in the region where the joint distribution of $(s_N^*, r_M^*)$ is concentrated. This can also be seen from Table \ref{tab:freq} where the mean value of $r_M^*/M$, which is computed conditionally to the realization of $s_N^*$, is always similar to or lower than that of $s_N^*/N$, and excessively large $r_M^*$ seldom occurs. Hence, these ``positive" validation outcomes from Bernoulli trials also take effect in reducing the conservatism of $\epsilon(s_N^*, r_M^*)$.


\begin{table*}
\caption{Occurrences of $s_N^*$ and $r_M^*$ in 2000 Monte Carlo Runs ($M = 1000$)}
\centering
\begin{tabular}{c | c | c c c c c c c}
\hline
& $s_N^*$ & 2 & 3 & 4 & 5 & 6 & 7 & 8 \\
\hline
          & Frequency & $45.70\%$ & $22.20\%$ & $23.30\%$ & $7.10\%$ & $1.50\%$ & $0.20\%$ & $0.00\%$ \\
$N = 100$ & $s_N^*/N$ & $2.00\%$ & $3.00\%$ & $4.00\%$ & $5.00\%$ & $6.00\%$ & $7.00\%$ & $8.00\%$ \\
          & Conditional mean($r_M^*/M$) & $1.90\%$ & $1.87\%$ & $1.92\%$ & $2.12\%$ & $2.45\%$ & $2.88\%$ & - \\
\hline
          & Frequency & $41.85\%$ & $21.70\%$ & $25.60\%$ & $9.15\%$ & $1.40\%$ & $0.25\%$ & $0.05\%$ \\
$N = 200$ & $s_N^*/N$ & $1.00\%$ & $1.50\%$ & $2.00\%$ & $2.50\%$ & $3.00\%$ & $3.50\%$ & $4.00\%$ \\
          & Conditional mean($r_M^*/M$) & $1.00\%$ & $0.90\%$ & $0.96\%$ & $1.08\%$ & $1.15\%$ & $1.37\%$ & $1.35\%$ \\
\hline
          & Frequency & $35.80\%$ & $24.75\%$ & $29.90\%$ & $7.70\%$ & $1.50\%$ & $0.35\%$ & $0.00\%$ \\
$N = 500$ & $s_N^*/N$ & $0.40\%$ & $0.60\%$ & $0.80\%$ & $1.00\%$ & $1.20\%$ & $1.40\%$ & $1.60\%$ \\
          & Conditional mean($r_M^*/M$) & $0.40\%$ & $0.36\%$ & $0.38\%$ & $0.44\%$ & $0.48\%$ & $0.76\%$ & - \\
\hline
\end{tabular}
\label{tab:freq}
\end{table*}

Finally, we consider the practical case where independent validation scenarios are successively accumulated after $\bold{x}_N^*$ is obtained. For clarity, a posteriori bounds based on 20 incrementally emerging scenarios have been depicted in Fig. \ref{fig:incremental}. Notice that tendencies of both $\epsilon_{N,M}(s_N^*,r_M^*)$ and $\eta_M(r_M^*)$ are consistent with constraint violation results, thereby confirming the correctness of Theorem \ref{thm:4}. Another observation is that the updating of $\epsilon_{N,M}(s_N^*,r_M^*)$ starts from $\epsilon_{N,0}(s_N^*,0) = 0.1176$, while the evolution of $\eta_M(r_M^*)$ has to start from $M=1$. Therefore, at the beginning stage, using $\epsilon_{N,M}(s_N^*,r_M^*)$ tends to be much more viable than using $\eta_M(r_M^*)$. This fundamentally owes to the usage of an increased level of information in $s_N^*$. It also justifies that, to reach a reliable certificate with $\eta_M(r_M^*)$ abundant samples are entailed \cite{tempo2012randomized}. In this sense, the proposed method enjoys desirable applicability in practical situations, especially when validation scenarios arrive in an incremental manner, or $M$ is not too large.
\begin{figure}[h]
\subfigure[]{
\includegraphics[width = 0.46\textwidth]{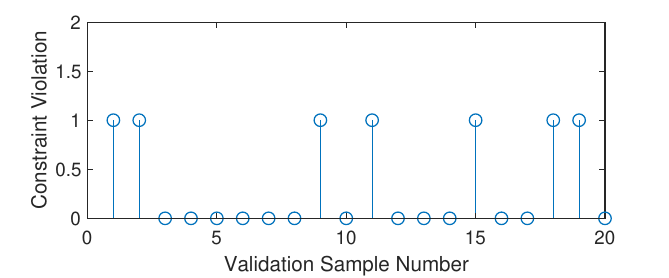}}
\subfigure[]{
\includegraphics[width = 0.485\textwidth]{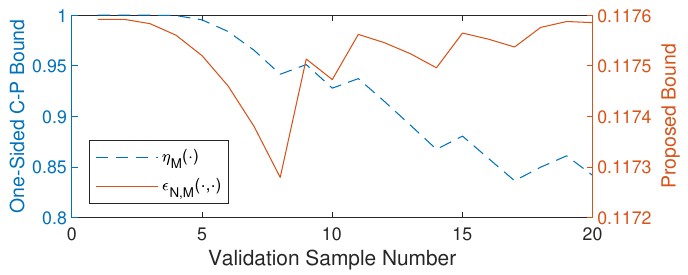}}
\caption{Incremental adjustment of a posteriori violation probabilities with $N = 200$, $s_N^* = 3$. (a) Constraint violations on validation samples. (b) Updated one-sided C-P bound $\eta_M(r_M^*)$ and the proposed bound $\epsilon_{N,M}(s_N^*,r_M^*)$. Notice that $\eta_M(r_M^*)$ and $\epsilon_{N,M}(s_N^*,r_M^*)$ have different ordinate scales.}
\label{fig:incremental}
\end{figure}

\section{Concluding Remarks}
In this work, we proposed a general class of a posteriori probabilistic bounds for convex scenario programs with violation probability of scenario-based solutions assessed on additional validation scenarios. The a posteriori bounds turned out to be a function of both the number of support constraints and violation frequency on validation datasets. It can assess the feasibility of scenario-based solutions with randomness overall considered, which arises from both sampling $N$ scenarios for formulating the problem and sampling $M$ scenarios for validation. Thanks to the utilization of more information, refined evaluation of solution's risk can be made by the proposed bounds. It has been shown the established result contains the existing bounds in scenario optimization as special cases, thereby enjoying wide generality. For practical use, we developed a class of a posteriori violation probabilities under prespecified confidence levels, which are shown to possess a variety of desirable properties allowing for easy implementations and clear interpretations. Comprehensive case studies demonstrated the efficacy of the proposed a posteriori bounds in judging robustness of solutions to convex scenario programs.

Finally, it is noteworthy that similar to the wait-and-judge optimization \cite{campi2018wait}, the proposed probabilistic guarantee does not imply a \textit{conditional validity} either, i.e., $\mathbb{P}^{N+M} \left \{ V(\bold{x}_N^*) > \epsilon(k,l) | s_N^* = k, r_M^* = l \right \} \le \beta^*$. The recent work \cite{garatti2019learning} provides an interesting idea for attaining such a conditional validity, which is worth further investigating. Another promising direction is to generalize these results in a repetitive design scheme to achieve a ``prescribed" risk level by deliberately choosing $N$ and $M$.




\section*{Acknowledgments}
The first author, C. Shang, was supported by National Science and Technology Innovation 2030 Major Project (No. 2018AAA0101604) of the Ministry of Science and Technology of China, and National Natural Science Foundation of China (Nos. 61873142, 61673236). F. You acknowledges financial support from the National Science Foundation (NSF) CAREER Award (CBET-1643244). The authors would like to thank the editor and anonymous reviewers for their valuable and constructive comments, which helped significantly improve this paper.



\end{document}